\documentclass[12pt]{amsart}

\usepackage[matrix,arrow,curve,frame]{xy}
\usepackage{amsmath,amsthm,amssymb,enumerate}
\usepackage{latexsym}
\usepackage{amscd}
\usepackage[colorlinks=false]{hyperref}
\usepackage{euscript}

\newtheorem{thm}[subsection]{Theorem}

\newtheorem{lemma}[subsection]{Lemma}

\theoremstyle{definition}

\numberwithin{equation}{section}

\def\cO{{\mathcal O}}

\DeclareMathOperator{\trace}{Trace}

\newfont{\german}{eufm10}

\begin{document}
\pagestyle{plain}

\title{ Chiral Hodge cohomology and Mathieu moonshine}

\author{Bailin Song}


\address{Key Laboratory of Wu Wen-Tsun Mathematics, Chinese Academy of
Sciences, School of Mathematical Sciences, University of Science and
Technology of China, Hefei, Anhui 230026, P.R. China}
\email{bailinso@ustc.edu.cn}

\thanks{The author is supported  by  NSFC 11771416.  }

\begin{abstract}
   We construct a filtration of chiral Hodge cohomolgy of a K3 surface $X$,  such that its associated graded object
 is a unitary representation of  the $\mathcal N=4$ superconformal vertex algebra  with central charge $c=6$ and its subspace of primitive vectors has the property: its equivariant character
 for  a finite symplectic automorphism $g$ of $X$ agrees with the  McKay-Thompson series for $g$ in Mathieu moonshine.
  \end{abstract}


\maketitle
\section{Introduction}
     In 2010, Eguchi, Ooguri and Tachikawa \cite{EOT} observed that when the elliptic genus of a K3 surface,
the Jacobi form $2 \phi_{0,1}(z;\tau)$ of weight $0$ and index $1$, is decomposed into a sum of the characters of
the $\mathcal N=4$ superconformal vertex algebra with central charge $c = 6$,
\begin{equation}\label{eq:chardecom} 2\phi_{0,1}(z;\tau)=-2ch_{1,\frac 1 4,\frac 1 2}^{\tilde R}(z;\tau) +20ch_{1,\frac 1 4,0}^{\tilde R}(z;\tau)+ 2\sum_{n=1}^\infty A_n ch_{1,n+\frac 1 4,\frac 1 2}^{\tilde R}(z;\tau),
\end{equation}
the first few coefficients $A_n$  are the sums of the dimensions of the irreducible representations of the largest Mathieu group $M_{24}$. Let
\begin{equation}\label{eq:graddim}\Sigma(q)=q^{-\frac 1 8}(-2+2\sum_{n=1}^\infty A_n q^n).
\end{equation} It is a mock
modular form of weight $\frac 1 2$. They conjectured that there exist   a
graded $M_{24}$-module $K = \sum_{n=0}^\infty K_n q^{n-1/8}$ with graded dimension $\Sigma(q)$.
It is Mathieu analogue to the modular function $J(q)$ in the famous monstrous moonshine\cite{B}\cite{CN}\cite{T}: the expansion coefficients of
$$
J(q) =\frac 1 q + 196884 q + 21493760 q^2 + \cdots  $$
could be naturally decomposed into the sums of the dimensions of the irreducible
representations of the largest sporadic group--the Fischer-Griess monster. Subsequently,  the McKay-Thompson series $\Sigma_g(q)$ for $g$ in $M_{24}$ were proposed
in several works \cite{C}\cite{CD}\cite{EH2}\cite{GHV}\cite{GHV2}
 and Terry Gannon \cite{G} proved that
these McKay-Thompson series indeed determine a graded $M_{24}$-module $K$.
But the proof does not explain any connection to geometry or physics, and a concrete construction of $K$ remains unknown.
In this work, we will construct a graded vector space from the chiral de Rahm algebra of the K3 surface with graded dimension $\Sigma(q)+2q^{-\frac 1 8}$.

  Chiral de Rham algebra $\Omega_X^{ch}$ is a sheaf of vertex algebras on a complex manifold $X$ constructed by Malikov, Schectman and Vaintrob in   ~\cite{MSV}\cite{MS}.
The sheaf has a $\mathbb Z\times \mathbb Z_{\geq 0}$ grading
$$\Omega_X^{ch}=\bigoplus_{k=0}^\infty\bigoplus_{p}\Omega_X^{ch}[k,p]$$
 by conformal weight $k$ and fermionic number $p$. And the weight zero piece of the sheaf
coincides with the ordinary de Rham sheaf. According to \cite{Kap}, If $X$ is a Calabi-Yau manifold, its cohomology $H^*(X,\Omega_X^{ch})$, which is called \textsl{chiral Hodge cohomology} ,
can be identified with the infinite-volume limit of the half-twisted sigma model defined by E. Witten. This construction has substantial applications
to mirror symmetry: Borisov established a relation between these sheaves of vertex algebras for mirror
Calabi-Yau hypersurfaces and complete intersections in toric varieties in \cite{Bor}.

If $X$ is a hyperK\"ahler manifold, we will show that $H^i(X,\Omega_X^{ch})$ has a filtration $\{H^i_k(X)\}$, such that its associated graded object
$\mathcal H^i(X,\Omega^{ch}_{X})=\oplus H^i_k(X)/H^i_{k+1}(X)$ is a unitary representation of  the $\mathcal N=4$ vertex algebra  with central charge $c=3\dim X$( Theorem \ref{thm:unitary}).

If $X$ is a K3 surface, let $\mathcal A_{n,2}^1(X)$ be the space of the primitive vectors with conformal weight $n$ in the unitary representation  $\mathcal H^1(X,\Omega^{ch}_{X})$
 of the $\mathcal N=4$ vertex algebra  with central charge $c=6$ .
Let
$$\mathcal A_X(q)=\sum_{n=1}^\infty \mathcal A_{n,2}^1(X)q^{n-\frac 1 8}.$$
We will show
\begin{thm} \label{thm:main}The  graded dimension of $\mathcal A_X(q)$ is $\Sigma(q)+2q^{\frac 1 8}$. For a finite symplectic automorphism $g$ acting on a K3
surface $X$,  $$\Sigma_g(q)+2q^{-\frac 1 8}=\sum_{n=1}^\infty q^{n-\frac 1 8}\trace_{\mathcal A^1_{n,2}(X)}g=\trace_{\mathcal A_X(q)}g.$$
Here $g$ is regarded as a member of $M_{24}$ by Mukai's classification of finite subgroups of  symplectic automorphisms of K3 surfaces in \cite{K}\cite{M}.
\end{thm}

Soon after this article was posted online, Katrin Wendland told me that she had similar result in \cite{W}, which was almost finished.

\section{Chiral de Rham complex}
The chiral de Rham algebra is a sheaf of vertex algebras $\Omega_X^{ch}$ defined on any complex manifold $X$. Let $(U,\gamma^1,\cdots \gamma^N)$ be a complex coordinate system of an $N$ dimensional complex manifold $X$. Let $\mathcal O(U)$ be the space of holomorphic functions on $U$.
As a vertex algebra $\Omega_X^{ch}(U)$ is generated by $\beta^i(z), b^i(z), c^i(z)$, $1\leq i\leq N$ and $f(z)$, $f\in \mathcal O(U)$ with their nontrivial OPEs
$$\beta^i(z)  f(w)\sim \frac {\frac{\partial f}{\partial \gamma^i}(z)}{z-w},\quad b^i(z) c^j(w)\sim \frac {\delta^i_j}{z-w}$$ and the normally ordered relations $$:f(z)g(z):\ =fg(z), \text{ for }  f, g \in \mathcal O(U).$$
Let $\Omega_\gamma$ be the sub algebra of $\Omega_X^{ch}(U)$ which is generated by  $\beta^i(z), b^i(z), c^i(z)$ and $\gamma^i(z)$, $1\leq i\leq N$. It is the tensor product of $N$ copies of the $\beta\gamma-bc$ system.
$\mathcal O(U)$ is a  $\mathbb C[\gamma^1_{(-1)},\cdots \gamma^N_{(-1)}]$ module by identifying the action $\gamma^i_{(-1)}$ with multiplication of $\gamma^i$ in $\mathcal O(U)$.
$\Omega_X^{ch}(U)$ can be regarded as the localization of $\Omega_{\gamma}$ on
$U$, $$\mathcal O(U)\otimes_{\mathbb C[\gamma^1_{(-1)},\cdots \gamma^N_{(-1)}]} \Omega_{\gamma}=\mathcal O(U)\otimes_{\mathbb C} \Omega_{\gamma}^+\cong \Omega_X^{ch}(U).$$
Here $\Omega_{\gamma}^+$ is the sub algebra of $\Omega_{\gamma}$ generated by $\beta^i(z),\partial \gamma^i(z), b^i(z)$ and $c^i(z)$. The second isomorphism maps $f\otimes a$ to $:f(z)a:$.

So $\Omega_X^{ch}(U)$ is spanned by the elements
\begin{eqnarray}\label{eq.span}
:f \partial^{k_1}\beta^{i_1}\cdots \partial^{k_s}\beta^{i_s}\partial^{l_1}b^{j_1}
\cdots \partial^{l_t}b^{j_t}\partial^{m_1}c^{r_1}
 \cdots\partial^{m_u}c^{r_u}\partial^{n_1}\gamma^{s_1}\cdots \partial^{n_v}\gamma^{s_v}:,\,\,\, f\in \mathcal O(U), \\
\quad k_1\geq k_2\geq \cdots\geq k_s\geq 0,\, l_1\geq\cdots \geq l_t\geq 0,\, m_1\geq\cdots \geq m_u\geq 0,\, n_1\geq\cdots \geq n_v>0.\nonumber
 \end{eqnarray}

 $\Omega_X^{ch}(U)$  is a free $\cO(U)$ module (which is not compatible with the vertex algebra structure), which has a basis given by elements of (\ref{eq.span}) with $f=1$.

$\Omega_X^{ch}[k,p](U)$, the part of conformal weights $k$ and fermionic numbers $p$  of $\Omega_X^{ch}(U)$, is a free $\cO(U)$ module, which has a basis given by elements of (\ref{eq.span})
with $f=1$ and $$k=\sum k_i+\sum l_i+\sum m_i+\sum n_i+s+t, \quad p=u-t.$$

 Let $\tilde \gamma^1,\cdots \tilde \gamma^N$ be another set of coordinates on $U$, with
\begin{equation}\label{eq:trans}\tilde \gamma^i=f^i(\gamma^1,\cdots \gamma^N), \quad \gamma^i=g^i(\tilde \gamma^1,\cdots \tilde \gamma^N).
\end{equation}
The coordinate change equations for the generators of $\Omega_X^{ch}(U)$ are
\begin{align}\label{chi.coo}
\partial \tilde \gamma^i(z)&=:\frac{\partial f^i}{\partial \gamma^j}(z)\partial \gamma^j(z):\,, \nonumber \\
\tilde b^i(z)&=:\frac{\partial g^j}{\partial \tilde \gamma^i}(g(\gamma))b^j: \nonumber\,, \\
\tilde c^i(z)&=:\frac{\partial f^i}{\partial \gamma^j}(z)c^j(z):\,,\\
\tilde \beta^i(z)&=:\frac{\partial g^j}{\partial \tilde \gamma^i}(g(\gamma))(z)\beta^j(z):
+::\frac{\partial}{\partial \gamma^k}(\frac{\partial g^j}{\partial \tilde \gamma^i}(g(\gamma)))(z)c^k(z):b^j(z):\,.\nonumber
\end{align}

 \subsection{Filtration}
 Let $\Omega_{X,l}^{ch}(U)$ be the subspace of $\Omega_X^{ch}(U)$, which is spanned by elements represented by(\ref{eq.span}) with $v-s\geq l$.
 $$\cdots \subset \Omega_{X,l+1}^{ch}(U)\subset \Omega_{X,l}^{ch}(U)\subset \Omega_{X,l-1}^{ch}(U)\subset\cdots
 $$
 is a filtration of $\Omega_{X}^{ch}(U)$.
 \begin{lemma}The filtration is compatible with the grading of conformal weight and fermionic number . And for $l>k$,
 $$\Omega_X^{ch}[k,p](U)\cap \Omega_{X,l}^{ch}(U)= \Omega_X^{ch}[k,p](U)\cap \Omega_{X,k}^{ch}(U);$$
 $$\Omega_X^{ch}[k,p](U)\cap \Omega_{X,-l}^{ch}(U)= \Omega_X^{ch}[k,p](U)\cap \Omega_{X,-k}^{ch}(U).$$
\end{lemma}
\begin{proof}By the definition,
$\Omega_X^{ch}[k,p](U)\cap \Omega_{X,l}^{ch}(U)$ is a free $\cO(U)$ module, which has a basis given by elements of (\ref{eq.span}) with $f=1$ and
$$k=\sum k_i+\sum l_i+\sum m_i+\sum n_i+s+t,\quad p=u-t, \quad v-s\geq l .$$
So $$ \Omega_{X,l}^{ch}(U)=\bigoplus_{k,p}\Omega_X^{ch}[k,p](U)\cap \Omega_{X,l}^{ch}(U).$$
We have the first statement.
For a fixed $k$, there is no elements represented by $(\ref{eq.span})$ with $v-s>k$ or $v-s<-k$ since $k\geq s$ and $k\geq v$, so we have the second statement.
\end{proof}
 \begin{lemma}\label{lem:prod} For  any $A\in\Omega_{X,k}^{ch}(U)$, $A'\in\Omega_{X,l}^{ch}(U)$,
 $A_{(n)}A'\in \Omega_{X,k+l}^{ch}(U).$ For any holomorphic function $f$ on $U$,
 $f_{(n)}A\in \Omega_{X,k+1}^{ch}(U)$, for  $n\neq -1$.
 \end{lemma}
 \begin{proof}We can assume $A$ and $A'$ are elements in the form of (\ref{eq.span}). $A_{(n)}A'$ is obtained  from $A$ and $A'$ by contracting some $\beta's$ with $\gamma's$ and some $b's$ with $c's$
  and acting  $\partial's$ on some  $\beta's$ , $\gamma's$ ,$b's$ and $c's$ in $A$
 and $A'$. One $\partial^s \gamma$, $s>0$  contracts with one $\beta$ and
 one $\beta $ contracts with one $\partial^s \gamma$, $s>0$ or a holomorphic function. Thus through the contraction, the number of $\partial^s\gamma's $,
  $s>0$ minus the number of $\beta's$ in $A$ and $A'$ will not decrease. The number of $\beta's$ will not change and the number of $\partial^s\gamma's$, $s>0$
 will not decrease by acting  $\partial's$ on  $\beta's$ , $\gamma's$, $b's$ and $c's$. So  $A_{(n)}A'\in \Omega_{X,k+l}^{ch}(U).$  If $A=f$ and $n\geq 0$, if $f_{(n)}A\neq 0$,  there is some $\beta$ contracts with $f$, so  $f_{(n)}A\in \Omega_{X,k+1}^{ch}(U)$. If $n<-1$, $f_{(n)}A=(\frac 1 {(-n-1)!}\partial^{-n-1}f)_{(-1)}A\in  \Omega_{X,k+1}^{ch}(U)$.
 \end{proof}
 Let $$ gr(\Omega_{X}^{ch})(U)=\bigoplus_k gr_k(\Omega_{X}^{ch})(U),\quad gr_k(\Omega_{X}^{ch})(U)=\Omega_{X,k}^{ch}(U)/\Omega_{X,k+1}^{ch}(U).$$
 Let
 $$p_k :\Omega_{X,k}^{ch}(U) \to  gr_k(\Omega_{X}^{ch})(U)$$
  be the projection.
  By Lemma \ref{lem:prod}, the $n$-th product
$$\__{(n)}\_: gr_k(\Omega_{X}^{ch})(U)\times gr_l(\Omega_{X}^{ch})(U)\to gr_{k+l}(\Omega_{X}^{ch})(U)$$
 given by
 $$p_k(A)_{(n)}p_l(A')=p_{k+l}(A_{(n)}A').$$
is well defined on $gr(\Omega_{X}^{ch})(U)$.  So $gr(\Omega_{X}^{ch})(U)$ is a vertex algebra with the $n$-th product given by the above equation. By Lemma \ref{lem:prod}, for any $f\in \mathcal O(U)$, any $A\in gr(\Omega_{X}^{ch})(U)$ and $n\neq -1$,  $f_{(n)} A=0$. So $gr(\Omega_{X}^{ch})(U)$ is
 a $\mathcal O(U)$ module under the $-1$-th product.

 Let $\mathbf B^i=p_{-1}(\beta^i)$, $\mathbf A^i=p_{1}(\partial\gamma^i)$, $\mathbf b^i=p_{0}(b^i)$, $\mathbf c^i=p_{0}(c^i)$.  We have
  \begin{equation}\label{eq:B0}\mathbf B_{(0)}=0\end{equation}
  on $gr(\Omega_{X}^{ch})(U)$ since $\beta_{0}$ maps $ \Omega_{X,k}^{ch}(U)$ to $\Omega_{X,k}^{ch}(U)$. The nontrivial OPEs of these elements are
 $$\mathbf B^i(z) \mathbf A^j(w)\sim \frac {\delta^i_j}{(z-w)^2},\quad \mathbf b^i(z) \mathbf c^j(w)\sim \frac {\delta^i_j}{z-w} .$$
 Let $\mathbf W_\gamma$ be the vertex algebra generated by these elements, then
 $gr(\Omega_{X}^{ch})(U)=\mathcal O(U) \otimes_{\mathbb C}\mathbf W_\gamma$. It is easy to see that as a vertex algebra $\mathbf W_\gamma$ is isomorphic to $\Omega_\gamma^+$ by identifying $\mathbf B^i,\mathbf A^i,\mathbf b^i$ and $\mathbf c^i$ with $\beta^i,\partial \gamma^i,b^i$ and $c^i$, respectively.

From Lemma \ref{lem:prod} and the coordinate change equations (\ref{chi.coo}), we have
\begin{lemma}\label{lem:prefilt}The filtration $\{\Omega_{X,k}^{ch}(U)\}$ is preserved under the coordination change.
\end{lemma}
Thus the vertex algebra sheaf $\Omega^{ch}_{X}$ has a filtration $\{\Omega^{ch}_{X,k}\}$. Its associated graded object
$$gr(\Omega^{ch}_{X})=\bigoplus \Omega^{ch}_{X,k}/\Omega^{ch}_{X,k+1}, \quad \text{with } \Omega^{ch}_{X,k}/\Omega^{ch}_{X,k+1}(U)=gr(\Omega^{ch}_{X,k})(U)$$
is a sheaf  of vertex algebra. Under the coordinate change (\ref{eq:trans}),
\begin{align}\label{chi.coo2}
\tilde {\mathbf A}^i&=:\frac{\partial f^i}{\partial \gamma^j} \mathbf A^j:\,, \nonumber \\
\tilde {\mathbf b}^i&=:\frac{\partial g^j}{\partial \tilde \gamma^i}(g(\gamma))\mathbf b^j: \nonumber\,, \\
\tilde {\mathbf c}^i&=:\frac{\partial f^i}{\partial \gamma^j}\mathbf c^j:\,,\\
\tilde {\mathbf B}^i&=:\frac{\partial g^j}{\partial \tilde \gamma^i}(g(\gamma))\mathbf B^j:\,.\nonumber
\end{align}
So $gr(\Omega^{ch}_{X})$ is the sheaf of the sections of holomorphic vector bundle
$$V_{q,y}=\bigotimes_{n=1}^\infty(S_{q^n}(T)\otimes S_{q^n}( T^*)\otimes \wedge_{y^{-1}q^n}T\otimes \wedge_{yq^{n-1}} T^*)=\sum _{k,p}V_{q,y}[k,p] y^pq^k$$
with  $gr(\Omega^{ch}_{X})[k,p]$, the part of conformal weight $k$ and fermionic number $p$ of $gr(\Omega^{ch}_{X})$, is the sheaf of the holomorphic sections of $V_{q,y}[k,p]$.
$V_{q,y}$ is a holomorphic vector bundle of vertex algebra and its fibre at $x\in U$ is isomorphic to $\mathbf W_\gamma$.

\subsection{Global sections}

If $X$ is a Calabi-Yau manifold,
  there are four global sections $Q(z)$, $L(z)$, $J(z)$ and $G(z)$ on $X$, which generate an $N=2$ super conformal vertex algebra with the central charge $c=3\dim X$.
 Locally,
\begin{align}\label{eqn:topvertex}
 &Q(z)= \sum_{i=1}^N:\beta^i(z)c^i(z):,& &L(z)=\sum_{i=1}^N(:\beta^i(z)\partial\gamma^i(z):-:b^i(z)\partial c^i(z):),& \\
&J(z)=-\sum_{i=1}^N:b^i(z)c^i(z):,& &G(z)=\sum_{i=1}^N:b^i(z)\partial\gamma^i(z):.&\nonumber
\end{align}
If $X$ has a nowhere vanishing holomorphic volume form $\omega_0$. Let $(U,\gamma^1,\cdots, \gamma^N)$ be a coordinate system of $X$ such that
$\omega_0=d\gamma^1\cdots d\gamma^N$.
There are two global sections $D(z)$ and $E(z)$  of  $\Omega_X^{ch}$, which can be constructed from $\omega_0$.
Locally,  $D(z)$ and $E(z)$ can be represented by
$$D(z)= : b^1(z)b^2(z)\cdots b^N(z):,\quad E(z)=:c^1(z)c^2(z)\cdots c^N(z):.$$
 Let $B(z)=Q(z)_{(0)}D(z)$,  $C(z)=G(z)_{(0)}E(z)$.  These eight sections $Q$, $J$, $L$, $G$, $E$, $D$, $B$, $C$ generate a vertex algebra $\mathcal V_0$. If $N=3$, it is  Odake's algebra\cite{O}.

Let $$\bar Q=p_{-1}(Q),\quad \bar L=p_{0}(L),\quad  \bar J=p_{0}(J),\quad  \bar G=p_{1}(G),$$
$$\bar D=p_{0}(D),\quad \bar E= p_{0}(E),\quad \bar B= p_{-1}(B),\quad \bar C=p_{1}(C).$$
These eight elements are holomorphic sections of $gr(\Omega^{ch}_{X})$. They generate a vertex algebra which is isomorphic to $\mathcal V_0$ through the isomorphism from $\mathbf W_\gamma$ to $\Omega_\gamma^+$.

If $X$ is a hyperK\"ahler manifold with a holomorphic symplectic form $\omega_1$, let $\omega_1^{-1}$ be the inverse bivector of $\omega_1$. Locally, they are given by
$\omega_1=\omega_{ij}d\gamma^i\wedge d\gamma^j$ and $\omega_1^{-1}=\omega^{ij}\frac{\partial}{\partial\gamma^i}\frac{\partial}{\partial \gamma^j}$.
Let
\begin{eqnarray*} E_1(z)=\sum_{ij}\omega_{ij}c^i(z)c^j(z),& \quad D_1(z)=\sum_{ij}\omega^{ij}b^i(z)b^j(z);\\
B_1(z)=Q_{(0)}D_1(z), & C_1=G(z)_{(0)}E_1(z).
\end{eqnarray*}
$Q, J, L, G, E_1, D_1, B_1, C_1$  are global sections of $\Omega_X^{ch}$ and they generate a copy of the $\mathcal N=4$  vertex algebra $\mathcal V_1$ of central charge
$c= 3\dim X$ \cite{H}. $\mathcal V_0$ is a sub algebra of $\mathcal V_1$.

Let $$\bar D_1=p_{0}(D_1),\quad \bar E_1= p_{0}(E_1),\quad \bar B_1= p_{-1}(B_1),\quad \bar C_1=p_{1}(C_1).$$
$\bar Q, \bar J, \bar L, \bar G, \bar E_1, \bar D_1, \bar B_1, \bar C_1$ are holomorphic sections of $gr(\Omega^{ch}_{X})$.
They generate a copy of the $\mathcal N=4$ vertex algebra which is isomorphic to $\mathcal V_1$ through the isomorphism from $\mathbf W_\gamma$ to $\Omega_\gamma^+$.

\section{Unitary representation}

Let $h=(-,-)_X$ be the Ricci flat K\"ahler metric on a compact Calabi-Yau maniflod $X$ with $H_{ij}=(\frac {\partial}{\partial \gamma^i}, \frac {\partial}{\partial \gamma^j})_X$.
\begin{lemma}\label{lem:Hermitianexist}
There is a unique   Hermitian metric $(-, -)$ on the vector bundle $V_{q,y}$, such that  for any $ a, b\in  gr(\Omega^{ch}_{X})(U)$,
\begin{eqnarray}\label{eqn:propherm}
&(1,1)=1;& \nonumber \\
&(\mathbf B_{(n)}^{i}a,b)=H_{ij}(a,\mathbf A^{j}_{(-n)}b), &\text{for any } n\in\mathbb Z , n\neq 0, ; \\
&(\mathbf b^{i}_{(n)}a,b)=H_{ij}(a, \mathbf c^{j}_{(-n-1)}b), &\text{for any } n\in\mathbb Z.\nonumber
\end{eqnarray}
\end{lemma}
\begin{proof}
For any holomorphic coordinate system $(U,\gamma^1,\cdots \gamma^N)$ of $X$,
let $\sum_jF_{ij}\frac {\partial}{\partial \gamma^j}$ be a smooth orthonormal frame of the tangent bundle on $U$, we have $\sum_{i,j}F_{ki}H_{ij}\bar F_{lj}=\delta_{kl}$.
Let $F^{jk}$ be the inverse matrix of $F_{ij}$, that is $\sum_j F_{ij}F^{jk}=\delta_{ik}$. Let
$$\tilde {\mathbf B}^k=F_{ki}\mathbf B^i,\quad \tilde {\mathbf b}^k=F_{ki}\mathbf b^i,\quad \tilde {\mathbf A}^k=F^{jk}\mathbf A^j,\quad \tilde {\mathbf c}^k=F^{jk}\mathbf c^j.$$
The equations of (\ref{eqn:propherm}) are equivariant to
\begin{eqnarray}\label{eqn:propherm2}
&(1,1)=1;& \nonumber \\
&(\tilde{\mathbf B}_{(n)}^{i}a,b)=\delta_{ij}(a,\tilde{\mathbf A}^{j}_{(-n)}b), &\text{for any } n\in\mathbb Z , n\neq 0, ; \\
&(\tilde {\mathbf b}^{i}_{(n)}a,b)=\delta_{ij}(a, \tilde{\mathbf c}^{j}_{(-n-1)}b), &\text{for any } n\in\mathbb Z.\nonumber
\end{eqnarray}
Let
$$\mathcal S=\{\sqrt{\frac{1}{k_1!k_2!\cdots k_l!}}X_1^{k_1}\cdot X_l^{k_l}1|X_1,\cdots, X_l \text{ are distinct elements of  } \frac 1{\sqrt{-n}}\tilde{\mathbf B}^i_{(n)}, \frac 1{\sqrt{-n}}\tilde{\mathbf A}^i_{(n)}, \tilde{\mathbf b}^i_{(n)}, \tilde{\mathbf c}^i_{(n)}, n<0 \}.$$
Then $\mathcal S$ is a smooth frame of the vector bundle $V_{q,y}$ on $U$. Let $(-, -)$ be the Hermitian metric for $V_{q,y}$ on $U$, such that the $\mathcal S$ is an orthonormal frame on $U$,
then this Hermitian metric satisfies equation (\ref{eqn:propherm2}), so as  equation (\ref{eqn:propherm}).
Since the fibre of $V_{q,y}$ at $x\in U$ is generated by $\mathbf B^i$, $\mathbf A^i$, $\mathbf b^i$, $\mathbf c^i$, by equation (\ref{eqn:propherm}), if the desired Hermitian metric exist, it must be unique.
Equations (\ref{eqn:propherm}) are compatible under coordinate change, so the Hermitian metric we choose does not dependent on the coordinate system we choose. Thus we get the unique Hermitian metric $(-, -)$ on the vector bundle $V_{q,y}$ with the desired property.
\end{proof}

Since $\mathbf B_{(0)}=0$ by (\ref{eq:B0}), through a direct calculation, we have
\begin{lemma}If $(-, -)$ is the Hermitian metric on $V_{q,y}$ given by Lemma \ref{lem:Hermitianexist}, then for any $ a, b\in  gr(\Omega^{ch}_{X})(U)$,
\begin{align}\label{eqn:conjugate1}
&(\bar Q_{(n)}a,b)=(a,\bar G_{(-n+1)}b), \quad \quad \quad\quad\quad\quad (\bar J_{(n)}a,b)=(a,\bar J_{(-n)}b),\nonumber\\
&(\bar L_{(n)}a,b)=(a,(\bar L_{(-n+2)}-(n-1)\bar J_{(-n+1)}b), \nonumber \\
&\text{if } X \text{ has a nonvanishing } N \text{ form},\nonumber\\
 &(\bar D_{(n)}a,b)=(a,(-1)^{\frac {N(N-1)}2 }\bar E_{(N-2-n)}b), \quad (\bar B_{(n)}a,b)=(a,(-1)^{\frac {(N-1)(N-2)}2 }\bar C_{(N-1-n)}b),\\
&\text{if } X \text{ is a hyperK\"ahler manifold},\nonumber\\
&(\bar {D_1}_{(n)}a,b)=(a,\bar {E_1}_{(-n)}b),\quad  \quad\quad\quad\quad\quad(\bar B_{1(n)}a,b)=(a,\bar C_{1(1-n)}b).\nonumber
\end{align}
\end{lemma}
Let $(-, -)_V$ be the Hermitian metric on $V_{q,y}\otimes \wedge\bar T^*X$ induced by the K\"ahler metric on $X$ and the Hermitian metric on $V_{q,y}$ in Lemma \ref{lem:Hermitianexist}.
So we have an Hermitian metric on the cohomology of $gr(\Omega^{ch}_{X})$ given by
$$(a, b)=\int_X(a,b)_V\Phi, .$$
Here $\Phi$ is the volume form of $X$ and  $ a, b\in H^*(X, gr(\Omega^{ch}_{X}))$  are represented by harmonic forms with values in $V_{q,y}$.

For a linear operator $O$ on  $H^*(X, gr(\Omega^{ch}_{X}))$, let $O^*$ be its conjugate operator, the operator satisfies
$$(O a,b)=(a, O^* b), \quad a, b \in H^*(X, gr(\Omega^{ch}_{X})). $$

\begin{lemma}\label{lem:unigrad}
If $X$ is a compact Calabi-Yau manifold, the cohomology $H^i(X, gr(\Omega^{ch}_{X}) )$ of the sheaf  $gr(\Omega^{ch}_{X})$ is a unitary representation of the $N=2$ surper conformal vertex algebra
 with \begin{align}\label{eqn:conjugate2}
&\bar Q_{(n)}^*= \bar G_{(-n+1)},& &\bar J_{(n)}^*= \bar J_{(-n)},&\\
&\bar L_{(n)}^*= \bar L_{(-n+2)}-(n-1)\bar J_{(-n+1)}.\nonumber
\end{align}
If $X$ has a nonvanishing holomorphic volume form, $H^i(X, gr(\Omega^{ch}_{X}) )$ is a unitary representation of $\mathcal V_0$ with equation (\ref{eqn:conjugate2}) and
\begin{equation}\label{eqn:conjugate3}\bar D_{(n)}^*= (-1)^{\frac {N(N-1)}2 }\bar E_{(N-2-n)},\quad \bar B_{(n)}^*= (-1)^{\frac {(N-1)(N-2)}2 }\bar C_{(N-1-n)}.
\end{equation}
If $X$ is a hyperK\"ahler manifold,  $H^i(X, gr(\Omega^{ch}_{X}) )$ is a unitary representation of $\mathcal V_1$ with equation (\ref{eqn:conjugate2}) and
\begin{equation} \label{eqn:conjugate4}\bar {D_1}^*_{(n)}=\bar {E_1}_{(-n)},\quad \bar {B_1}^*_{(n)}=\bar {C_1}_{(1-n)}.
\end{equation}
\end{lemma}
\begin{proof}
If $X$ is a HyperK\"ahler manifold, to show $H^i(X, gr(\Omega^{ch}_{X}) )$ is a unitary representation of $\mathcal V_1$,
 we only need to show equation (\ref{eqn:conjugate2}) and (\ref{eqn:conjugate4}) since $\mathcal V_1$ is generated by $\bar Q, \bar J, \bar L, \bar G, \bar E_1, \bar D_1, \bar B_1, \bar C_1$.
It is easy to see that the $n$-th production of these eight elements with  any harmonic form with value in $V_{q,y}$ is still harmonic. Equations (\ref{eqn:conjugate1}) are still satisfied
for the Hermitian metric on $H^*(X, gr(\Omega^{ch}_{X}))$ since they are satisfied for the Hermitian metric on $V_{q,y}$.  So we have equation (\ref{eqn:conjugate2}) and (\ref{eqn:conjugate4}).

The proof for other cases is similar.
\end{proof}

\subsection{A filtration of  chiral Hodge cohomology}

Let $$\tau_k: \Omega^{ch}_{X,k} \to \Omega^{ch}_{X}$$ be the imbedding, which induces the morphism of their cohomology
$$\tau_{k*}: H^i(X, \Omega^{ch}_{X,k}) \to  H^i(X,\Omega^{ch}_{X}).$$
Let $ H^i_k(X)$ be the image of $\tau_{k*}$. $\{ H^i_k(X)\}$ is a filtration of $H^i(X,\Omega^{ch}_{X})$. Let  $\mathcal H^i(X,\Omega^{ch}_{X})=\bigoplus_k H^i_k(X)/H^i_{k+1}(X)$ be its associated graded object.
\begin{thm}\label{thm:unitary} If $X$ is a compact Calabi-Yau manifold, $H^i(X,\Omega^{ch}_{X})$ has a filtration $\{H^i_k(X)\}$, such that
$\mathcal H^i(X,\Omega^{ch}_{X})$ is a unitary representation of the $N=2$  vertex algebra with central charge $3\dim X$; if $X$ has a nonvanishing holomorphic volume form,
$\mathcal H^i(X,\Omega^{ch}_{X})$ is a unitary representation of  $\mathcal V_0$;  if $X$ is a hyperK\"ahler manifold,
$\mathcal H^i(X,\Omega^{ch}_{X})$ is a unitary representation of  the $\mathcal N=4$  vertex algebra $\mathcal V_1$ with central charge $c=3\dim X$.
\end{thm}
\begin{proof}
 Assume $X$ is a hyperK\"ahler manifold.
 $\bigoplus_k \Omega^{ch}_{X,k}$  has   a $\mathcal V_1$ module structure given by
\begin{align}
 Q_{(n)}a,\, {B_1}_{(n)}a  &&\in \Omega^{ch}_{X,k-1}(U) &;\nonumber \\
 L_{(n)}a,\, J_{(n)}a,\, {D_1}_{(n)}a,\, {E_1}_{(n)}a  & &\in\Omega^{ch}_{X,k}(U);& \hspace{1cm}\text {for any } a\in \Omega^{ch}_{X,k}(U)\nonumber \\
 G_{(n)}a,\, {C_1}_{(n)}a  &&\in \Omega^{ch}_{X,k+1}(U)&.\nonumber
 \end{align}
The exact sequences of sheaves
$$0\to \Omega^{ch}_{X,k+1} \xrightarrow{\iota_{k}} \Omega^{ch}_{X,k}\to \Omega^{ch}_{X,k}/\Omega^{ch}_{X,k+1}\to 0$$
gives an exact sequence of $\mathcal V_1$ modules
$$0\to \bigoplus_k\Omega^{ch}_{X,k+1} \to \bigoplus_k \Omega^{ch}_{X,k}\to gr(\Omega^{ch}_{X})\to 0,$$
which induces a long exact sequence of $\mathcal V_1$ modules
\begin{equation}\label{eq:exact}\cdots \to \bigoplus_k H^i(X,\Omega^{ch}_{X,k+1}) \to \bigoplus_k H^i(X, \Omega^{ch}_{X,k})\to H^i(X, gr(\Omega^{ch}_{X}))\to\cdots.
\end{equation}
 We have an imbedding of $\mathcal V_1$ modules
$$\bigoplus_k H^i(X, \Omega^{ch}_{X,k})/ \iota_{k*}H^i(X, \Omega^{ch}_{X,k+1})\to H^i(X, gr(\Omega^{ch}_{X})).$$
In particular,  $\bigoplus_k H^i(X, \Omega^{ch}_{X,k})/ \iota_{k*}H^i(X, \Omega^{ch}_{X,k+1})$ is a unitary representation of the vertex algebra $\mathcal V_1$, since  by Lemma \ref {lem:unigrad},
$H^i(X, gr(\Omega^{ch}_{X}))$
is a unitary representation of $\mathcal V_1$.

 $\bigoplus_k \Omega^{ch}_{X}$ has a  $\mathcal V_1$  module structure, which is similar to the $\mathcal V_1$ module structure of $ \bigoplus_k \Omega^{ch}_{X,k}$.
 Under these $\mathcal V_1$ module structure,
$$\bigoplus_k \Omega^{ch}_{X,k} \to \bigoplus_k \Omega^{ch}_{X}$$
is a morphism of $\mathcal V_1$ modules.
 So
$$\bigoplus_kH^i(X, \Omega^{ch}_{X,k})\to \bigoplus_k H^i_k(X)$$
is a morphism of $\mathcal V_1$ module. We get a morphism of $\mathcal V_1$ module
$$\tau_*: \bigoplus_kH^i(X, \Omega^{ch}_{X,k})/\iota_{k*}H^i(X, \Omega^{ch}_{X,k+1})\to \bigoplus_k H^i_k(X)/H^i_{k+1}(X).$$
By the definition of $H^i_k(X)$, $\tau_* $ is surjective.  $\bigoplus_kH^i(X, \Omega^{ch}_{X,k})/\iota_{k*}H^i(X, \Omega^{ch}_{X,k})$ is unitary representation of $\mathcal V_1$,
 so $\bigoplus_k H^i_k(X)/H^i_{k+1}(X)$ is a unitary representation of $\mathcal V_1$.

 The proofs for the other two cases are similar.
\end{proof}
Since the filtration $\Omega^{ch}_{X,k}$ is compatible with the grading of conformal weights and fermionic numbers,
the filtration $H^i_k(X)$ is also compatible with the grading of conformal weights and fermionic numbers. In particular,the filtration on the piece of weight zero is trivial,
we have $$\mathcal H^i(X,\Omega^{ch}_{X}[0,p])=H^i(X,\Omega^{ch}_{X}[0,p]).$$

 For the rest of the paper, we assume $X$ is a hyperK\"ahler manilfold. By Theorem \ref{thm:unitary},
 $\mathcal H^i(X,\Omega^{ch}_{X})$ is a direct sum of irreducible representation of the $\mathcal N=4$  vertex algebra with central charge $c=3\dim X$.

\subsection{Unitary representation of $\mathcal V_1$}
For an irreducible unitary representation $M_{k,h,l}$ of the $\mathcal N=4$ vertex algebra $\mathcal V_1$ with central charge $c=6k$,
the  highest weight vector $v$ of $M_{k,h,l}$ is labeled by conformal weight $h$ and the isospin $l$.
$$L_{(1)}v =(h-\frac k 4)v,\quad J_{(0)}v=(2l+k)v.$$

There exist two types of representations in the $N = 4$ vertex algebra\cite{ET1}\cite{ET2}, massless (BPS) ($h=\frac k 4$, $l=0,\frac 1 2,1,\cdots, \frac k 2
$) and massive (non-BPS) ($h>\frac k 4$, $l=\frac 1 2,1,\cdots, \frac k 2
$) representations.
 The character of a representation $V$ of the $\mathcal N=4$ vertex algebra for Ramond sectors  is  defined by
$$ch_{V}^{R}(z;\tau)=y^{-k}\trace_{V}y^{J_{(0)}}q^{L_{(1)}},$$
and
$$ch_{V}^{\tilde R}(z;\tau)=ch_{V}^{R}(z+\frac 1 2;\tau)=(-y)^{-k}\trace_{V}(-y)^{J_{(0)}}q^{L_{(1)}},$$
where $q=e^{2\pi i \tau}, y=e^{2\pi iz}$.
Let  $ch_{k,h,l}^{\tilde R}(z;\tau)$ be the character of the representation $M_{k,h,l}$.

Let $\mathcal A_{h,l}^i(X)$ be the subspace of $\mathcal H^i(X,\Omega^{ch}_{X})$, which consist primitive vectors $v$ with
$$\bar L_{(1)}v=hv; \quad  \bar L_{(n)}v=0, n>1;$$
$$\bar J_{(0)}v=lv; \quad \bar J_{(n)}v=0, n>0;$$
$$\bar Q_{(n)}v=0, n\geq 0;\quad  \bar G_{(n)}v=0, n>1;$$
$$\bar {E_1}_{(n)}v=0, n\geq -1;\quad  \bar {D_1}_{(n)}v=0, n> 1;$$
$$\bar {B_1}_{(n)}v=0, n>1;\quad  \bar {C_1}_{(n)}v=0, n\geq0.$$
By Theorem \ref{thm:unitary},
\begin{equation}\label{eq:hdecom}\mathcal H^i(X,\Omega^{ch}_{X})
=(\bigoplus_{l=0}^kM_{k,\frac k 4,\frac l 2}\otimes \mathcal A^i_{0,k+l }(X))\bigoplus(\bigoplus_{n=1}^\infty\bigoplus_{l=1}^kM_{k,n
+\frac k 4,\frac l 2}\otimes \mathcal A^i_{n,k+l}(X)).
\end{equation}
Let $A^i_{n,l}=A^i_{n,l}(X)=\dim \mathcal A^i_{n,l}(X)$.
$$ch^{\tilde R}_{\mathcal H^i(X,\Omega^{ch}_{X})}(z;\tau)=\sum_{l=0}^kA^i_{0,l+k}ch_{k,\frac k 4,\frac l 2}^{\tilde R}(z;\tau)+\sum_{n=1}^\infty\sum_{l=1}^kA^i_{n,l+k}ch_{k,n
+\frac k 4,\frac l 2}^{\tilde R}(z;\tau).$$

\subsection{The complex elliptic genus}
For a holomorphic vector bundle $E$  on $X$, its  Euler characteristic is
$\chi(X,E)=\sum_{i=0}^N(-1)^i H^i(X,E)$.
The complex elliptic genus of $X$ is
$$Ell_X(z;\tau)=y^{-\frac{\dim X}{2}}\chi(X, V_{-y,q}).$$

By \cite{BL} or by the long exact sequence (\ref{eq:exact}), it is equal to the graded dimension of the cohomology
of the chiral de Rham algebra of $X$, i,e
  $$Ell_X(z;\tau)=y^{-\frac{\dim X}{2}}\sum_{i=0}^n(-1)^i\trace_{H^i(X,\Omega^{ch}_{X})}(-y)^{J_{(0)}}q^{L_{(1)}}.$$
If $g$ is an automorphism of the holomorphic vector bundle $E$, let
$$\chi(g; X,E)=\sum_{i=0}^N(-1)^i \trace_{H^i(X,E)} g.$$
For an automorphism $g$ of $X$, by the long exact sequence (\ref{eq:exact}), the equivariant elliptic genus of $X$  is
$$Ell_{X,g}(z;\tau)=y^{-\frac{\dim X}{2}}\chi(g;X, V_{-y,q}).$$
It is equal to
\begin{equation}\label{eq:eqell}Ell_{X,g}(z;\tau)=y^{-\frac{\dim X}{2}}\sum_{i=0}^N(-1)^i\trace_{H^i(X,\Omega^{ch}_{X})}g(-y)^{J_{(0)}}q^{L_{(1)}}.
\end{equation}
A complex automorphism $g$ of the hyperK\"ahler manifold $X$ is called symplectic if it preserve the holomorphic symplectic 2-form $\omega_1$, i.e. $g^*\omega_1=\omega_1$. If $g$ is symplectic automorphism,
from the definition of the eight generators of $\mathcal V_1$, the elements of $\mathcal V_1$ are $g$ invariant. So $g$ acts on $\mathcal A^i_{n,l+k}(X)$.
We have
\begin{thm}\label{thm:char} If $X$ is a hyperK\"ahler manifold, $g$ is a symplectic automorphism of $X$,
\begin{eqnarray}\label{eqn:equiell}Ell_{X,g}(z;\tau)=(-1)^{k}\sum_{i=0}^N(-1)^i(\sum_{l=0}^k ch_{k,\frac k 4,\frac l 2}^{\tilde R}(z;\tau) \trace_{\mathcal A^i_{0,l+k}(X)}g\\
+\sum_{n=1}^\infty\sum_{l=1}^k ch_{k,n+\frac k 4,\frac l 2}^{\tilde R}(z;\tau) \trace_{\mathcal A^i_{n,l+k}(X)}g).\nonumber
\end{eqnarray}
In particular
$$Ell_{X}(z;\tau)
=(-1)^{k}\sum_{i=0}^N(-1)^i(\sum_{l=0}^k A^i_{0,l+k} ch_{k,\frac k 4,\frac l 2}^{\tilde R}(z;\tau)
+\sum_{n=1}^\infty\sum_{l=1}^k A^i_{n,l+k} ch_{k,n+\frac k 4,\frac l 2}^{\tilde R}(z;\tau) ).$$
\end{thm}
\begin{proof}
$H^i_k$ is a filtration of $H^i(X,\Omega^{ch}_X)$ and it is compatible with the grading of the conformal weight and fermionic number, so
$$\trace_{\mathcal H^i(X,\Omega^{ch}_{X})}g(-y)^{\bar J_{(0)}}q^{\bar L_{(1)}}=\trace_{H^i(X,\Omega^{ch}_{X})}g(-y)^{J_{(0)}}q^{L_{(1)}}.$$
By equation (\ref{eq:eqell}),
$$Ell_{X,g}(z;\tau)=y^{-\frac{\dim X}{2}}\sum_{i=0}^n(-1)^i\trace_{\mathcal H^i(X,\Omega^{ch}_{X})}g(-y)^{\bar J_{(0)}}q^{\bar L_{(1)}}.$$
Thus by equation (\ref{eq:hdecom}) and the fact that elements of $\mathcal V_1$ is $g$ invariant,
we get equation (\ref{eqn:equiell}).
\end{proof}

\section{Relation to Mathieu Moonshine}
In this section we explain the relation between the Mathieu Moonshine and chiral Hodge cohomologe of the K3 surfaces.

\subsection{K3 surface}  Let $X$ be a K3 surface,  the dimensions of its Dolbeault cohomology groups are
$\dim H^{2,2}(X)=1,$
 $\dim H^{2,1}(X)=0$ and $\dim H^{1,1}(X)=20.$

In \cite{S1}\cite{S2}, we showed that $H^0(X,\Omega_X^{ch})$ is the simple $\mathcal N=4$ vertex algebra $\mathcal V_1$ with central charge $c=6$.
So $H^0(X,\Omega_X^{ch})$ and $\mathcal H^0(X,\Omega_X^{ch})$ are isomorphic to $M_{1,\frac 1 4, \frac 1 2}$.
By chiral Poincar\'e duality\cite{MS2},
 $$H^0(X,\Omega_X^{ch}[k,p])\cong H^2(X,\Omega_X^{ch}[k,2-p]).$$
 So
  $$\bigoplus_p H^0(X,\Omega_X^{ch}[k,p])\cong \bigoplus_p H^2(X,\Omega_X^{ch}[k,2-p]).$$
 Since $H^2(X,\Omega_X^{ch})$ is a module of $\mathcal N=4$ vertex algebra with central charge $c=6$,
 $H^2(X,\Omega_X^{ch})$ and  $\mathcal H^2(X,\Omega_X^{ch})$ are isomorphic to $M_{1,\frac 1 4, \frac 1 2}$.

 Obviously
   $$A_{0,2}^1(X)=\mathcal H^1(X,\Omega_X^{ch}[0,2])=H^1(X,\Omega_X^{ch}[0,2])=H^{2,1}(X)=0,$$
  $$\mathcal A^1_{0,1}(X)=\mathcal H^1(X,\Omega_X^{ch}[0,1])=H^1(X,\Omega_X^{ch}[0,1])= H^{1,1}(X).$$
 By (\ref{eq:hdecom}),
 $$\mathcal H^1(X,\Omega_X^{ch})=M_{1,\frac 1 4,0}\otimes H^{1,1}(X)\bigoplus (\bigoplus_{n\geq 1}M_{1,n+\frac 1 4,\frac 1 2} \otimes \mathcal A^1_{n,2}(X)).$$
 By Theorem \ref{thm:char}, we get equation (\ref{eq:chardecom}), the decomposition of the elliptic genus of K3 surface,
$$2\phi_{0,1}(z;\tau)=Ell_{X}(z;\tau)
=-2ch_{1,\frac 1 4,\frac 1 2}^{\tilde R}(z;\tau) +20ch_{1,\frac 1 4,0}^{\tilde R}(z;\tau)+ 2\sum_{n=1}^\infty A_n ch_{1,n+\frac 1 4,\frac 1 2}^{\tilde R}(z;\tau).  $$
Here $A_n= \frac 1 2 A^1_{n,2}(X)$.

Let
$$\mathcal A_X(q)=\sum_{n=1}^\infty \mathcal A_{n,2}^1(X)q^{n-\frac 1 8}.$$
$\mathcal A_X(q)$ has graded dimension $\Sigma(q)+2q^{-\frac 1 8}$.

Since $\mathcal H^2(X,\Omega_X^{ch})$ is isomorphic to $M_{1,\frac 1 4, \frac 1 2}$,
the space of the primitive vectors of $\mathcal H^2(X,\Omega_X^{ch})$ is $H^2(X,\Omega_X^{ch}[0,2])=H^{2,2}(X)$.
If $g$ is a symplectic automorphism of $X$,  the primitive vectors of $\mathcal H^2(X,\Omega_X^{ch})$  are  $g$ invariant.
So
 \begin{equation}\label{eq:eqellk3decom} Ell_{X,g}(z;\tau)
=-2ch_{1,\frac 1 4,\frac 1 2}^{\tilde R}(z;\tau) +ch_{1,\frac 1 4,0}^{\tilde R}(z;\tau)\trace_{H^{1,1}(X)}g
+ \sum_{n=1}^\infty \trace_{\mathcal A_{n,2}^1(X)}g\, ch_{1,n+\frac 1 4,\frac l 2}^{\tilde R}(z;\tau).
\end{equation}
\subsection{Proof of Theorem \ref{thm:main}}
The McKay-Thompson series for $g$ in $M_{24}$ are
\begin{equation}\label{eq:twistchar}\Sigma_g(q)=q^{-\frac 1 8}\sum_{n=0}^\infty q^n\trace_{K_n}g=\frac{e(g)}{24}\Sigma_e(q)-\frac{f_g}{\eta(q)^3},
\end{equation}
where $\Sigma_e(q)=\Sigma(q)$,
$e(g)$ is the character of the $24$-dimensional permutation representation of $M_{24}$, the
series $f_g$ is a certain explicit modular form of weight $2$ for some subgroup $\Gamma^0(N_g)$
of $SL(2,\mathbb Z)$ and $\eta$ is the Dedekind eta function.

There is a deep relation between the K3 surfaces and the Mathieu group $M_{24}$. Mukai has classified the finite symplectic automorphism groups of K3 surfaces in \cite{K}\cite{M},
which are all isomorphic to subgroups of the Mathieu group $M_{23}$ of a
particular type. $M_{23}$ is isomorphic to a one-point stabilizer for the
permutation action of $M_{24}$ on $24$ elements.
If $g$ is a symplectic automorphism of a K3 surface, we can regard $g$ as an element of $M_{24}$.
 Thomas Creutzig and Gerald H\"ohn in \cite{CH} showed that
 for a non-trivial finite symplectic automorphism $g$ acting on a K3
surface $X$, the equivariant elliptic genus and the twining character determined by the
McKay-Thompson series of Mathieu moonshine agree, i.e. one has
\begin{equation}\label{eqn:ellagree}Ell_{X,g}(z;\tau)=\frac {e(g)}{12}\phi_{0,1}+f_g\phi_{-2,1}.\end{equation}
If $g$ is a finite symplectic automorphism acting on a K3
surface $X$, by equation equation (\ref{eq:eqellk3decom}), equation (\ref{eq:twistchar}) and (\ref {eqn:ellagree}), we immediately get Theorem \ref{thm:main}.

By Gannon's theorem , $\mathcal A_X(q)$ is a graded $M_{24}$ module. One may ask whether it is possible to construct a concrete action of the Mathieu group $M_{24}$ on $\mathcal A_X(q)$?

\end{document}